\documentclass{amsart}

\usepackage{amsmath,amssymb}
\usepackage{amsthm}
\usepackage{indentfirst}
\usepackage{mathrsfs}
\usepackage{graphicx}
\usepackage{xcolor}
\usepackage{fourier}
\usepackage{microtype}
\usepackage[margin=1.5in]{geometry}
\usepackage{algorithm}
\usepackage{algorithmic}
\usepackage{booktabs}
\floatstyle{ruled}
\restylefloat{algorithm}

\usepackage{todonotes}
\usepackage{hyperref}

\theoremstyle{plain}
\newtheorem{theorem}{Theorem}
\newtheorem{lemma}[theorem]{Lemma}      

\newtheorem{prop}[theorem]{Proposition}

\theoremstyle{definition}

\theoremstyle{remark}
\newtheorem{remark}[theorem]{Remark}    

\newcommand{\ud}{\,\mathrm{d}}

\newcommand{\RR}{\mathbb{R}}

\newcommand{\EE}{\mathbb{E}}

\IfFileExists{mathabx.sty}%
  {\DeclareFontFamily{U}{mathx}{\hyphenchar\font45}%
   \DeclareFontShape{U}{mathx}{m}{n}{<->mathx10}{}%
   \DeclareSymbolFont{mathx}{U}{mathx}{m}{n}%
   \DeclareFontSubstitution{U}{mathx}{m}{n}%
   \DeclareMathAccent{\widebar}{0}{mathx}{"73}%
}{%
  \PackageWarning{mathabx}{%
    Package mathabx not available, therefore\MessageBreak substituting
    widebar with overline\MessageBreak }%
  \newcommand{\widebar}[1]{\overline{#1}}%
}

\newcommand{\abs}[1]{\lvert#1\rvert}

\usepackage{enumitem}

\newcommand{\R}{\RR}
\newcommand{\E}{\EE}

\newcommand{\dd}{\ud}

\newcommand{\ip}[2]{\langle #1,#2\rangle}
\newcommand{\pos}[1]{\left(#1\right)_+}

\DeclareMathOperator{\Law}{Law}

\newcommand{\bps}{{\rm BPS}}
\newcommand{\zz}{{\rm ZZ}}

\title[Windowed thinning for BPS and Zigzag]
      {Windowed thinning and query complexity for the bouncy particle and Zigzag samplers}
\thanks{This work is supported in part by the National Science Foundation under grant DMS-2309378. We thank Sinho Chewi and Lihan Wang for helpful discussions.}

\author{Jianfeng Lu}
\address{Department of Mathematics, Department of Physics, and Department of Chemistry, Duke University}
\email{jianfeng@math.duke.edu}

\author{Yinchen Luo}
\address{Department of Mathematics, Duke University}
\email{yinchen.luo@duke.edu}
\date{September 1, 2026}


\begin{document}

\begin{abstract}
Let $\mu(\dd x)\propto e^{-U(x)}\dd x$ on $\R^d$, where $U$ is $m$-strongly convex and $L$-smooth, and denote by $\kappa=L/m$ the condition number.  We consider windowed thinning, an exact simulation method for the bouncy particle sampler and the coordinate Zigzag process.  The method divides a trajectory into deterministic windows and uses a gradient evaluation at the beginning of each window to construct a tractable local envelope for the event rate.  Combining this construction with quantitative mixing estimates and finite-time bounds on the expected numbers of bounces and flips yields query complexity guarantees from a Gaussian cold start.  For total-variation error $\varepsilon$, the expected query counts are $O(\kappa^{1/2}d\,(d\log\kappa+\log\frac1\varepsilon))$ gradient queries for the bouncy particle sampler and $O(\kappa d^{1/4}(d\log\kappa+\log\frac1\varepsilon))$ full-gradient equivalents for Zigzag, where $d$ coordinate-partial queries count as one equivalent.  
\end{abstract}

\maketitle

\section{Introduction}\label{sec:intro}

The bouncy particle sampler (BPS) and the Zigzag sampler offer an event-driven geometric approach to sampling.  A particle follows a simple deterministic path and changes its velocity only at random event times.  If those event times are generated exactly, the resulting continuous-time trajectory incurs no time-discretization error.  This attractive feature leaves a less visible computational question: how many evaluations of the target gradient are needed to run the process until it is close to equilibrium?  Answering this question requires controlling both the time to convergence and the cost of simulating the trajectory over that time.  Our goal is to provide such a query complexity guarantee from an explicit Gaussian cold start.

We consider the log-concave target measure
\[
        \mu(\dd x)= Z^{-1}e^{-U(x)}\dd x,
        \qquad x\in\R^d,
\]
where $Z$ is the normalizing constant and $U\in C^2(\R^d)$ is
$m$-strongly convex and $L$-smooth:
\begin{equation}\label{eq:hessian-bounds}
        m I_d \preceq \nabla^2 U(x) \preceq L I_d,
        \qquad x\in\R^d.
\end{equation}
We assume that $m$, $L$, and the minimizer $x_\star$ of $U$ are available,
together with oracle access to either the full gradient $\nabla U(x)$ or an
individual coordinate derivative $\partial_iU(x)$.

Within this setting, we focus on two samplers based on piecewise deterministic
Markov processes (PDMPs): BPS \cite{BouchardCoteVollmerDoucet} and Zigzag
\cite{BierkensFearnheadRoberts}.  BPS reverses the component of its velocity
in the local gradient direction, whereas Zigzag flips a single velocity
coordinate at each event.
Both processes evolve on
$\R^d\times\R^d$ and preserve the product measure
\[
        \rho_\infty(\dd x\dd v)=\mu(\dd x)\varphi(\dd v),
        \qquad \varphi:=N(0,I_d).
\]  
See Section~\ref{sec:prelim} for a more detailed description.
We initialize them from an explicit Gaussian distribution centered at
$x_\star$, \textit{i.e.,} from a cold start.

Turning either continuous-time process into an exact algorithm requires
generating its bounce or flip times.  Their event rates depend on $\nabla U$
along the evolving trajectory and cannot be evaluated continuously in the
oracle model.  Poisson thinning \cite{LewisShedler1979} replaces the true
event rate by a tractable upper envelope: candidate events are proposed from
the envelope and accepted with probability equal to the ratio of the true
rate to the envelope.  A valid envelope preserves exactness, but a loose one can
generate many rejected proposals and therefore many unnecessary queries.

Our main algorithmic contribution, \emph{windowed thinning}, resolves this
tradeoff locally.  We divide the simulation horizon into deterministic
windows and evaluate the gradient at the current position at the beginning
of each window.  The $L$-Lipschitz continuity of $\nabla U$ then supplies a local envelope for
the event rate in terms of the cumulative distance travelled from this anchor.
Short windows keep the envelope tight but require frequent anchor queries;
long windows require fewer anchors but generate more rejected proposals.
Balancing these two query counts leads to the window lengths used in our
bounds.  Combined with quantitative
mixing estimates, windowed thinning gives the following expected query counts:
\begin{align*}
\text{BPS:}\quad&
        O\!\left(\kappa^{1/2}d
        \left(d\log\kappa+\log\tfrac1\varepsilon\right)\right),\\
\text{Zigzag:}\quad&
        O\!\left(\kappa d^{1/4}
        \left(d\log\kappa+\log\tfrac1\varepsilon\right)\right).
\end{align*}
Here $\kappa:=L/m$ is the condition number and $\varepsilon$
denotes the desired total-variation accuracy.
The BPS bound counts full-gradient queries, while the Zigzag bound is stated in
full-gradient equivalents, with $d$ coordinate-partial queries counted as one
equivalent.  Both algorithms simulate their underlying continuous-time
processes exactly; the accuracy parameter only determines the simulation
horizon.

Table~\ref{tab:complexity-comparison} compares our bounds with those of other
high-accuracy samplers whose dependence on the target accuracy is
polylogarithmic.  All entries include the cost of starting from a feasible
cold distribution; for the Metropolis-adjusted Langevin algorithm (MALA),
this includes an algorithmic warm-start phase.  The oracle and guarantee conventions are recorded in the table caption. Compared with MALA and first-order rejection sampling (FORS), while our bound for BPS has a smaller $\kappa$ dependence, its dimension dependence is much worse. We also note that since Zigzag is a coordinate algorithm, it is possible to reduce the complexity under additional structural assumptions, such as assuming coordinate Lipschitz conditions with better constants; we do not pursue those in this work.  

\begin{table}[H]
\caption{Cold-start query complexities of high-accuracy samplers for
$m$-strongly convex, $L$-smooth targets, stated for total-variation error at
most $\varepsilon$.  All entries include a feasible initialization.  The MALA
entry counts first-order queries returning both $U$ and $\nabla U$; the FORS
and PDMP entries count expected gradient queries, with Zigzag counts expressed
in full-gradient equivalents.  The notation $\widetilde O$ suppresses
logarithmic factors in $d$, $\kappa$, and $\varepsilon^{-1}$.}
\label{tab:complexity-comparison}
\small
\centering
\begin{tabular}{@{}p{0.52\textwidth}p{0.25\textwidth}@{}}
\toprule
Method & Oracle-query complexity \\
\midrule
MALA with algorithmic warm start \cite{AltschulerChewi2024}
 & $\widetilde O(\kappa d^{1/2})$ \\
Proximal sampler with FORS \cite{ChenChewiDaskalakisRakhlin2026}
 & $\widetilde O(\kappa d^{1/2})$\\
Windowed-thinning BPS, this work
 & $\widetilde O(\kappa^{1/2}d^2)$ \\
Windowed-thinning Zigzag, this work
 & $\widetilde O(\kappa d^{5/4})$ \\
\bottomrule
\end{tabular}
\end{table}

\medskip

To establish these query counts, we analyze convergence and simulation
separately.  Quantitative $\chi^2$-contraction estimates, recalled in
Section~\ref{ssec:LW}, first determine how long the process must run.  Over
this horizon, the thinning analysis counts the anchor evaluations and
proposed events.

The main difficulty is to control the event count directly from the nonstationary Gaussian cold start, without assuming stationarity or a warm start. Rather than tracking individual bounces and flips, we bound their expected total number through the time-integrated squared event rates. For each sampler, a Dynkin identity for a suitable observable bounds these integrals in terms of position- and velocity-dependent quantities, which are controlled using the convexity and smoothness of the potential together with the cold-start moment estimates. This turns a path-dependent counting problem into a moment estimate. The velocity-moment bounds also control the distance travelled within each window, and hence the expected number of rejected proposals. Balancing this rejection cost against the number of anchor evaluations determines the window length. Together, these estimates yield dimension-explicit query bounds over the full simulation horizon. The details are in Section~\ref{sec:proof}.

\subsection*{Related work}

The event-driven bounce mechanism underlying BPS originated with Peters and
de With
\cite{PetersDeWith2012} and was developed into the BPS framework by
Bouchard-C\^ot\'e, Vollmer and Doucet
\cite{BouchardCoteVollmerDoucet}.  A one-dimensional Zigzag process was
obtained as a scaling limit of lifted Metropolis--Hastings for the Curie--Weiss magnetization
\cite{BierkensRoberts2017} and was later developed as a general-purpose
sampler by Bierkens, Fearnhead and Roberts
\cite{BierkensFearnheadRoberts}.  Both belong to the broader PDMP Monte Carlo
family surveyed in \cite{FearnheadEtAl2018}.  Later examples include the
Coordinate Sampler, the Boomerang Sampler, and Forward Event-Chain Monte Carlo
\cite{WuRobertCoordinate,BierkensEtAlBoomerang,MichelDurmusSenecal}.

Exact trajectory simulation requires sampling event times from inhomogeneous
Poisson processes.  Classical thinning \cite{LewisShedler1979} and adaptive
thinning \cite{Ogata1981} underlie the local construction used in the original
BPS work \cite{BouchardCoteVollmerDoucet}.  Subsequent work develops analytic
envelopes \cite{SuttonFearnhead2023} and practical numerical bounding procedures
\cite{CorbellaSpencerRoberts2022,AndralKamatani2024}. Our work builds on this literature with a quantitative cold-start analysis.

Harris-type methods establish broad ergodicity results for these samplers.
For BPS, Deligiannidis, Bouchard-C\^ot\'e and Doucet proved geometric
ergodicity and a central limit theorem under explicit tail and curvature
conditions \cite{DeligiannidisBouchardCoteDoucet2019}, while Durmus, Guillin
and Monmarch\'e treated broader classes of targets and velocity laws using
coupling and quantitative minorization
\cite{DurmusGuillinMonmarche2020}.  For the Zigzag process,
Bierkens, Roberts and Zitt proved total-variation convergence, and also established exponential convergence and a central limit theorem for empirical averages under stronger tail assumptions \cite{BierkensRobertsZitt2019}. 
From the scaling-limit perspective, Bierkens, Kamatani and Roberts analyzed BPS
and Zigzag in the limit of $d \to \infty$ for a standard Gaussian target
\cite{BierkensKamataniRoberts2022}, and in a fixed-dimensional
extreme-anisotropy limit $\epsilon \to 0$ for Gaussian targets
\cite{BierkensKamataniRoberts2025}.  Agrawal,
Bierkens, Kamatani and Roberts study the transient motion from low-probability
regions toward a typical set using fluid limits, with expected event counts as
a proxy for computational cost
\cite{AgrawalBierkensKamataniRoberts2025}.  

Explicit nonasymptotic results quantify convergence, finite-time fluctuations,
and query counts.  Andrieu, Durmus, N\"usken and Roussel developed a general
$L^2$-hypocoercivity framework covering BPS and Zigzag
\cite{AndrieuDurmusNueskenRoussel2021}.  Lu and Wang derived explicit
dimension- and condition-number-dependent rates \cite{LuWangPDMP}, based on the space-time Poincare inequality approach \cite{CaoLuWangUnderdamped}; these are
the continuous-time inputs used here.  Their Zigzag complexity analysis
combines such a rate with a global event envelope to obtain a high-probability
coordinate-partial query bound under a warm-start assumption
\cite{LuWangZigzag}.  Beyond contraction rates,
Birrell and Rey-Bellet derived Bernstein-type concentration inequalities for
finite-time ergodic averages and uncertainty-quantification bounds
\cite{BirrellReyBellet2019}.  From the lifting viewpoint, Eberle and L\"orler interpret BPS and Zigzag as nonreversible lifts and show that any second-order lift can improve relaxation time by at most a square-root factor \cite{EberleLorlerLifts}.

Recent works also consider convergence on the entropy level. Under convexity, a spatial log--Sobolev inequality and suitable growth conditions, and with friction $\gamma=\Gamma\sqrt{\rho}$, Lu proves exponential relative-entropy decay at the sharp order $\sqrt{\rho}$ for every initial law with finite relative entropy \cite{LuUnderdampedEntropy}. Under analogous assumptions and a tame Hessian-growth condition, Li and Lu establish a space--time log--Sobolev inequality and hypocoercive hypercontractivity, yielding exponential decay of R\'enyi divergences of every order $q>1$ at the sharp hypocoercive order $\sqrt{\rho}$ \cite{LiLuHypercontractivity}. Unfortunately, for BPS and Zigzag with positive refreshment and compactly supported or standard Gaussian velocity laws, Monmarch\'e and Wang show that, even for a standard Gaussian target, no nontrivial uniform relative-entropy contraction over all initial laws in any relative-entropy neighborhood of equilibrium holds \cite{MonmarcheWangEntropy}; thus our analysis is based on $L^2$-hypocoercivity instead.

For comparison beyond PDMP samplers, Wu, Schmidler and Chen proved the optimal
warm-start complexity $\widetilde O(\kappa d^{1/2})$ for MALA
\cite{WuSchmidlerChen2022} (the earlier bound was
$\widetilde O(\kappa d+\kappa^{3/2}\sqrt d)$
\cite{DwivediChenWainwrightYu2019}).  Altschuler and Chewi
\cite{AltschulerChewi2024} obtain the same complexity from a feasible start by
combining a R\'enyi warm start constructed using discretized underdamped
Langevin dynamics and warm-started MALA within a proximal-sampler reduction.
Chen, Chewi, Daskalakis and Rakhlin
\cite{ChenChewiDaskalakisRakhlin2026} implement the proximal sampler using
gradient-only FORS and match this expected complexity for the same target
class.

With polynomial dependence on the target accuracy, Altschuler, Chewi and Zhang
\cite{AltschulerChewiZhang2026} analyze a randomized-midpoint discretization
of underdamped Langevin dynamics for smooth strongly log-concave targets, obtaining
$\widetilde O(\kappa^{5/6}d^{5/3}\varepsilon^{-2/3})$ gradient evaluations
from a Gaussian cold start to achieve KL error at most $\varepsilon^2$. Combining their discretization analysis with the
space--time log--Sobolev inequality 
\cite{LiLuHypercontractivity} yields
$\widetilde O(\kappa^{5/6}d^{1/3}\varepsilon^{-2/3})$ gradient evaluations for total-variation error at most $\varepsilon$, under the regularity conditions of the latter result
\cite[Theorem~5.3.17]{ChewiLogConcaveSampling}.

For lower bound of query complexity, Chewi, de Dios Pont, Li, Lu and Narayanan \cite{ChewiEtAlQueryLB} prove that,
for sufficiently large $d$, sampling even from
centered Gaussian targets to constant total-variation accuracy requires
$\Omega(\kappa^{1/2}\log d)$ first-order queries.  Consequently, in
this regime the $\kappa$ dependence of our BPS query bound is optimal, up to logarithmic factors.

Finally, concurrent work by Chen, Chewi, Rakhlin and Zhang
\cite{ChenChewi} develops a high-accuracy sampler by applying rejection
sampling to entire underdamped Langevin paths.
Combining the FORS mechanism \cite{ChenChewiDaskalakisRakhlin2026,ChenChewiDaskalakisRakhlinStochastic2026} with the space--time log--Sobolev inequality and hypocoercive
hypercontractivity \cite{LuUnderdampedEntropy, LiLuHypercontractivity}, they obtain expected
first-order-query complexity $\widetilde O_q(\kappa^{2/3}d^{1/3})$ for
achieving error at most $\varepsilon$ in R\'enyi divergence of order $q$.

\subsection*{Notation}

Throughout, $\abs{\cdot}$ is the Euclidean norm on $\R^d$,
$\abs{\cdot}_1$ is the $\ell^1$ norm, and $(a)_+:=\max\{a,0\}$.
For probability measures $\nu\ll\pi$, we write
\[
        \chi^2(\nu\Vert\pi)
        :=\int\left(\frac{\dd\nu}{\dd\pi}-1\right)^2\dd\pi .
\]
The notation $\widetilde O$ suppresses factors polylogarithmic in
$d$, $\kappa$, and $\varepsilon^{-1}$.  A gradient query evaluates
$\nabla U(y)$, while a coordinate-partial query evaluates
$\partial_iU(y)$; $d$ coordinate-partial queries count as one full-gradient
equivalent.  The letter $C$ denotes a universal constant whose
value may change from line to line.

\subsection*{Use of AI tools} The general algorithm design and analysis approach are created by the authors. LLM-based assistants were used in the preparation of this manuscript for drafting, language editing, and consistency checks.  The authors are responsible for independently checking every statement and proof and for the correctness and integrity of the final manuscript. Upon finishing the draft, we have been informed by Fan Chen, Sinho Chewi and Yiping Lu that ChatGPT Pro 5.6 was able to produce the result for BPS (before our draft becomes public) and in fact establish a high probability guarantee. We decided not to include the stronger result in the manuscript. 

\section{Preliminaries}\label{sec:prelim}

\subsection{Bouncy particle sampler}\label{ssec:bps}

The bouncy particle sampler (BPS) is a piecewise deterministic Markov process
on $\R^d\times\R^d$ with deterministic flow
$\dot X_t=V_t$, $\dot V_t=0$ between events.

For $\nabla U(x)\neq 0$, define the reflection
\begin{equation}\label{eq:reflection}
R_x v
:=
v-2\frac{\langle v,\nabla U(x)\rangle}
{\abs{\nabla U(x)}^2}\nabla U(x),
\end{equation}
and set $R_x v=v$ if $\nabla U(x)=0$.  At state $(x,v)$, bounces occur with
rate
\[
\lambda^{\bps}(x,v)
:=
\pos{\ip{v}{\nabla U(x)}},
\]
and each bounce replaces $v$ by $R_xv$.  Independently, refreshment events
occur at constant rate $\gamma^{\bps}>0$, and each refresh replaces $v$ by an
independent draw from $\varphi$.

Accordingly, the generator acts on sufficiently smooth $f$ as
\begin{equation}\label{eq:bps-generator}
\mathcal L^{\bps} f(x,v)
=
v\cdot\nabla_x f(x,v)
+
\lambda^{\bps}(x,v)\bigl(f(x,R_x v)-f(x,v)\bigr)
+
\gamma^{\bps}\left(
  \int_{\R^d}f(x,w)\varphi(\dd w)-f(x,v)
\right).
\end{equation}
The BPS process has only finitely many events on bounded time intervals and
has invariant law $\rho_\infty=\mu\otimes\varphi$
\cite{BouchardCoteVollmerDoucet}.

Note that $R_x$ is orthogonal, so
it preserves $\abs v$ and the Gaussian law $\varphi$.  It also reverses the
sign of $\ip{v}{\nabla U(x)}$; consequently, the speed $\abs{V_t}$ changes
only at refreshment times.  

\subsection{Zigzag sampler}\label{ssec:zz}

The Zigzag sampler is a piecewise deterministic Markov process on
$\R^d\times\R^d$ with the same deterministic flow
$\dot X_t=V_t$, $\dot V_t=0$ between events.  Let
$F_i v:=v-2v_i e_i$ denote the velocity obtained by flipping coordinate $i$.
At state $(x,v)$, the $i$th velocity coordinate flips with rate $\lambda_i^{\zz}(x,v)$, and the total coordinate-flip rate is
$\Lambda^{\zz}(x,v)$, where
\begin{equation}\label{eq:zz-rates}
\lambda_i^{\zz}(x,v):=\pos{v_i\partial_iU(x)},
\qquad
\Lambda^{\zz}(x,v):=\sum_{i=1}^d\lambda_i^{\zz}(x,v).
\end{equation}
At a coordinate-$i$ event, $v$ is replaced by $F_i v$.  Independently,
refreshment events occur at constant rate $\gamma^{\zz}>0$ and replace the
entire velocity by an independent draw from $\varphi$.  The corresponding
generator is
\begin{equation}\label{eq:zz-generator}
\mathcal L^{\zz}f(x,v)
=
v\cdot\nabla_xf(x,v)
+
\sum_{i=1}^d\lambda_i^{\zz}(x,v)
\bigl(f(x,F_i v)-f(x,v)\bigr)
+
\gamma^{\zz}\left(
  \int_{\R^d}f(x,w)\varphi(\dd w)-f(x,v)
\right).
\end{equation}
The Zigzag process likewise has only finitely many events on bounded time
intervals and has invariant law $\rho_\infty=\mu\otimes\varphi$
\cite{BierkensFearnheadRoberts,LuWangZigzag}.

Each $F_i$ preserves $\abs{v_j}$ for every $j$, and hence preserves both
$\abs v$ and the Gaussian law $\varphi$; it also reverses the sign of
$v_i\partial_iU(x)$.  Thus, as for BPS, the speed changes only at refreshment
times.  

\subsection{Quantitative \texorpdfstring{$\chi^2$}{chi-square} contraction estimates}\label{ssec:LW}

Our analysis uses the following quantitative contraction estimates of Lu and
Wang, stated with the refreshment rates explicit.

\begin{theorem}[Lu--Wang \cite{LuWangPDMP}, BPS]\label{thm:LW}
Assume \eqref{eq:hessian-bounds} and set $\gamma^{\bps}=\sqrt{dm}$.
There exists a universal constant $K^{\bps}\ge1$ such that, for every initial
law $\rho_0\ll\rho_\infty$ with
$\chi^2(\rho_0\Vert\rho_\infty)<\infty$, BPS initialized from $\rho_0$
satisfies
\begin{equation}\label{eq:LW-L2}
        \chi^2\!\left(
          \Law(X_t,V_t)\,\middle\Vert\,\rho_\infty
        \right)
        \le K^{\bps}
        \exp\left(-\frac{1}{K^{\bps}}\sqrt{\tfrac md}\,t\right)
        \chi^2(\rho_0\Vert\rho_\infty),
        \qquad t\ge0 .
\end{equation}
\end{theorem}

\begin{theorem}[Lu--Wang \cite{LuWangZigzag}, Zigzag]\label{thm:LW-ZZ}
Assume \eqref{eq:hessian-bounds} and set $\gamma^{\zz}=\sqrt L$.  There exists
a universal constant $K^{\zz}\ge1$ such that, for every initial law
$\rho_0\ll\rho_\infty$ with
$\chi^2(\rho_0\Vert\rho_\infty)<\infty$, Zigzag initialized from $\rho_0$
satisfies
\begin{equation}\label{eq:LW-ZZ}
        \chi^2\!\left(
          \Law(X_t,V_t)\,\middle\Vert\,\rho_\infty
        \right)
        \le K^{\zz}
        \exp\!\left(-\frac{m}{K^{\zz}\sqrt L}\,t\right)
        \chi^2(\rho_0\Vert\rho_\infty),
        \qquad t\ge0.
\end{equation}
\end{theorem}

Throughout, we use the explicit Gaussian cold start
\begin{equation}\label{eq:cold-start}
        \rho_0= N(x_\star,L^{-1}I_d)\otimes N(0,I_d).
\end{equation}
The Hessian bounds yield the density comparison
\begin{equation}\label{eq:cold-start-chi2}
        \frac{\dd\rho_0}{\dd\rho_\infty}\le\kappa^{d/2},
        \qquad
        \chi^2(\rho_0\Vert\rho_\infty)\le\kappa^{d/2}-1.
\end{equation}
Thus Theorems~\ref{thm:LW} and~\ref{thm:LW-ZZ} apply.  The
exponentially large initial $\chi^2$ divergence enters the mixing horizons
through its logarithm, and contributes a factor of $d$ in the complexity bound.

\section{Windowed thinning and complexity theorems}\label{sec:wt}

The convergence estimates of Section~\ref{sec:prelim} specify how long each
process must run.  This section combines those continuous-time guarantees with
exact simulation algorithms to obtain end-to-end complexity bounds. We first introduce windowed Poisson thinning and
present its BPS and Zigzag implementations, then state the resulting
end-to-end complexity theorems.  The detailed query estimates are deferred to
Section~\ref{sec:proof}.

Along each continuous-time process, events arrive at a time-varying Poisson
rate.  Thinning generates proposals from an upper rate, or envelope, determined
by the trajectory observed so far.  A proposal is accepted with probability
equal to the ratio of the true rate to this envelope.  If the envelope bounds
the true rate at every time, the accepted events have exactly the desired law.

Windowed thinning keeps this envelope local.  We divide the simulation horizon
into deterministic windows and evaluate the gradient at the particle's
current position at the beginning of each window.   The $L$-Lipschitz continuity of $\nabla U$ then controls
how far the true event rate can move from this anchored estimate as the
particle travels.  The window length sets a natural tradeoff: short windows
use more anchor evaluations but keep the envelope tight, whereas long windows
cause the envelope to loosen and
generate more rejected proposals.  We choose the window length by balancing
these two effects.

\subsection{Windowed thinning for BPS}\label{ssec:wt-bps}

Recall that for BPS the bounce rate is
$\pos{\ip{V_{s-}}{\nabla U(X_{s-})}}$.  Here and below, $s-$ denotes the left
limit, since an event may occur at time $s$.
At the beginning of the window $[t_k,t_{k+1})$, we know the anchor gradient
$G_k=\nabla U(X_{t_k})$, but not its future values along the path.  Let
$D_s:=\int_{t_k}^s\abs{V_r}\dd r$ denote the cumulative distance travelled from the
anchor.  Smoothness controls the deviation:
\[
        \abs{\nabla U(X_{s-})-G_k}
        \le L\abs{X_{s-}-X_{t_k}}
        \le L D_s,
\]
which motivates the envelope
\begin{equation}\label{eq:bps-envelope}
        \bar\lambda_s
        :=\pos{\ip{V_{s-}}{G_k}}+L\abs{V_{s-}}D_s.
\end{equation}
Indeed, it bounds the true rate at every time:
\begin{align*}
        \lambda^{\bps}(X_{s-},V_{s-})
        &=\pos{\ip{V_{s-}}{\nabla U(X_{s-})}}\\
        &\le\pos{\ip{V_{s-}}{G_k}}
             +\abs{\ip{V_{s-}}{\nabla U(X_{s-})-G_k}}\\
        &\le\pos{\ip{V_{s-}}{G_k}}+L\abs{V_{s-}}D_s
        =\bar\lambda_s.
\end{align*}

Between state changes and window boundaries, the velocity is constant and
$D_s$ is affine.  Hence $\bar\lambda_s$ is affine on each such interval, and
its time integral is quadratic, so the next proposal time can be found by
solving a quadratic equation.
Algorithm~\ref{alg:bps} gives the resulting procedure, incorporating this
envelope together with the refreshment clock and the next window boundary.
In both algorithms, $\operatorname{Exp}(a)$ denotes the exponential
distribution with rate $a$.

\begin{algorithm}[htb]
\caption{Windowed thinning for BPS}\label{alg:bps}
\begin{algorithmic}[1]
\REQUIRE horizon $T>0$, window length $\tau>0$, refreshment rate
         $\gamma^{\bps}$, initial state $(X_0,V_0)$
\STATE $t\gets0$; partition $[0,T]$ into windows $[t_k,t_{k+1})$,
       $t_k=k\tau$, the last window truncated at $T$
\FOR{each nonempty window $[t_k,t_{k+1})$}
  \STATE query the anchor gradient $G_k\gets\nabla U(X_{t_k})$; set
         $D_{t_k}\gets0$
  \WHILE{$t<t_{k+1}$}
    \STATE draw $E\sim\operatorname{Exp}(1)$; generate the next proposal time
           $t_{\rm p}>t$ by solving
           $\int_t^{t_{\rm p}}\bar\lambda_s\dd s=E$, using
           \eqref{eq:bps-envelope}; set $t_{\rm p}=\infty$ if there is no
           solution before $t_{k+1}$
    \STATE draw $t_{\rm r}\gets t+\operatorname{Exp}(\gamma^{\bps})$
           independently
    \STATE $t\gets\min\{t_{\rm p},t_{\rm r},t_{k+1}\}$; advance $X$ and $D$
           deterministically to time $t$
    \IF{$t=t_{\rm p}$}
      \STATE query $\nabla U(X_{t-})$; draw
             $A\sim\operatorname{Unif}(0,1)$
      \IF{$A\le\lambda^{\bps}(X_{t-},V_{t-})/\bar\lambda_t$}
        \STATE $V_t\gets R_{X_{t-}} V_{t-}$, using the queried gradient
      \ELSE
        \STATE $V_t\gets V_{t-}$
      \ENDIF
    \ELSIF{$t=t_{\rm r}$}
      \STATE draw $V_t\sim\varphi$, retaining $G_k$ and $D$
    \ENDIF
  \ENDWHILE
\ENDFOR
\RETURN $(X_T,V_T)$
\end{algorithmic}
\end{algorithm}

The envelope bound above and standard Poisson thinning show that
Algorithm~\ref{alg:bps} simulates BPS exactly.  A proposal can occur only when
$\bar\lambda_t>0$, so its acceptance ratio is well defined.  The complete
proof that the construction is well defined and exact is given in
Proposition~\ref{prop:bps-thinning}.

\subsection{Windowed thinning for Zigzag}\label{ssec:wt-zz}

Recall that for Zigzag, the flip rate of coordinate $i$ is given by
\[
        \lambda_i^{\zz}(X_{s-},V_{s-})
        =\pos{V_{i,s-}\partial_iU(X_{s-})}.
\]
At the beginning of the window $[t_k,t_{k+1})$, we know the anchor gradient
$G_k=\nabla U(X_{t_k})$, but not the future coordinate derivatives along the
path.  As in the BPS case, let
$D_s:=\int_{t_k}^s\abs{V_r}\dd r$ denote the cumulative distance travelled from the
anchor.  Then, for every coordinate $i$,
\[
        \abs{\partial_iU(X_{s-})-G_{k,i}}
        \le\abs{\nabla U(X_{s-})-G_k}
        \le L\abs{X_{s-}-X_{t_k}}
        \le L D_s.
\]
This motivates the coordinatewise envelopes and their sum
\begin{equation}\label{eq:zz-envelope}
        \bar\lambda_{i,s}
        :=\pos{V_{i,s-}G_{k,i}}+L\abs{V_{i,s-}}D_s,
        \qquad
        \bar\Lambda_s:=\sum_{i=1}^d\bar\lambda_{i,s}.
\end{equation}
Indeed, each envelope bounds its coordinate's true rate:
\begin{align*}
        \lambda_i^{\zz}(X_{s-},V_{s-})
        &=\pos{V_{i,s-}\partial_iU(X_{s-})}\\
        &\le\pos{V_{i,s-}G_{k,i}}
             +\abs{V_{i,s-}}\,
              \abs{\partial_iU(X_{s-})-G_{k,i}}\\
        &\le\pos{V_{i,s-}G_{k,i}}+L\abs{V_{i,s-}}D_s
        =\bar\lambda_{i,s}.
\end{align*}
Summing over the coordinates gives
$\Lambda^{\zz}(X_{s-},V_{s-})\le\bar\Lambda_s$.  The coordinate proposal
clocks can therefore be combined into a single clock with rate
$\bar\Lambda_s$.  When it rings, coordinate $i$ is selected
with probability $\bar\lambda_{i,s}/\bar\Lambda_s$, and only
$\partial_iU(X_{s-})$ needs to be queried to decide whether that flip is
accepted.

Between state changes and window boundaries, the velocity is constant and
$D_s$ is affine.  Hence the total envelope $\bar\Lambda_s$ is affine on each
such interval, and its time integral is quadratic, so the next proposal time
can be found by solving a quadratic equation.  Algorithm~\ref{alg:zz} gives
the resulting procedure,
incorporating this envelope together with the refreshment clock and the next
window boundary.

\begin{algorithm}[htb]
\caption{Windowed thinning for Zigzag}\label{alg:zz}
\begin{algorithmic}[1]
\REQUIRE horizon $T>0$, window length $\tau>0$, refreshment rate
         $\gamma^{\zz}$, initial state $(X_0,V_0)$
\STATE $t\gets0$; partition $[0,T]$ into windows $[t_k,t_{k+1})$,
       $t_k=k\tau$, the last window truncated at $T$
\FOR{each nonempty window $[t_k,t_{k+1})$}
  \STATE query the anchor gradient $G_k\gets\nabla U(X_{t_k})$; set
         $D_{t_k}\gets0$
  \WHILE{$t<t_{k+1}$}
    \STATE draw $E\sim\operatorname{Exp}(1)$; generate the next proposal time
           $t_{\rm p}>t$ by solving
           $\int_t^{t_{\rm p}}\bar\Lambda_s\dd s=E$, using
           \eqref{eq:zz-envelope}; set $t_{\rm p}=\infty$ if there is no
           solution before $t_{k+1}$
    \STATE draw $t_{\rm r}\gets t+\operatorname{Exp}(\gamma^{\zz})$
           independently
    \STATE $t\gets\min\{t_{\rm p},t_{\rm r},t_{k+1}\}$; advance $X$ and $D$
           deterministically to time $t$
    \IF{$t=t_{\rm p}$}
      \STATE draw coordinate $i$ with probability
             $\bar\lambda_{i,t}/\bar\Lambda_t$
      \STATE query $\partial_iU(X_{t-})$; draw
             $A\sim\operatorname{Unif}(0,1)$
      \IF{$A\le\lambda_i^{\zz}(X_{t-},V_{t-})/\bar\lambda_{i,t}$}
        \STATE $V_t\gets F_i V_{t-}$
      \ELSE
        \STATE $V_t\gets V_{t-}$
      \ENDIF
    \ELSIF{$t=t_{\rm r}$}
      \STATE draw $V_t\sim\varphi$, retaining $G_k$ and $D$
    \ENDIF
  \ENDWHILE
\ENDFOR
\RETURN $(X_T,V_T)$
\end{algorithmic}
\end{algorithm}

The envelope bounds above and standard Poisson thinning show that
Algorithm~\ref{alg:zz} simulates Zigzag exactly.  A coordinate can be
proposed only when $\bar\lambda_{i,t}>0$, so its acceptance ratio is well
defined.  The complete proof that the construction is well defined and exact
is given in
Proposition~\ref{prop:zz-thinning}.

\subsection{Main complexity theorems}\label{ssec:main-thms}

For Algorithm~\ref{alg:bps}, let $Q_\nabla(T)$ denote the expected number of
full-gradient queries up to time $T$.  For Algorithm~\ref{alg:zz}, let
$Q_\partial(T)$ denote the expected number of coordinate-partial queries and
set $Q_{\rm eq}(T):=Q_\partial(T)/d$ for the corresponding number of
full-gradient equivalents.  These quantities measure oracle calls only; a
direct implementation of Zigzag also uses $O(d)$ arithmetic operations per
proposal to form and sample from the envelope weights.

Combining the mixing estimates of Section~\ref{sec:prelim} with the
fixed-horizon query estimates proved in Section~\ref{sec:proof} gives
end-to-end guarantees from the Gaussian cold start.  We state the BPS and
Zigzag results in parallel, giving the latter in both coordinate-partial and
full-gradient-equivalent units.

\begin{theorem}[Cold-start BPS complexity]\label{thm:main}
Suppose that \eqref{eq:hessian-bounds} holds and initialize BPS according to
\eqref{eq:cold-start}.  Algorithm~\ref{alg:bps} is exact.  Set
$\gamma^{\bps}=\sqrt{dm}$.  For $\varepsilon\in(0,1/2)$, define
\begin{equation}\label{eq:bps-horizon}
        \widehat T_\varepsilon^{\bps}
        :=\max\left\{
          2K^{\bps}\sqrt{\frac dm}\,
          \log\!\left(
            1+\frac{\sqrt{K^{\bps}\chi^2(\rho_0\Vert\rho_\infty)}}
                     {2\varepsilon}\right),
          L^{-1/2},\frac{\gamma^{\bps}}{4L}\right\}.
\end{equation}
At time $\widehat T_\varepsilon^{\bps}$, the joint $\chi^2$ error is at most
$4\varepsilon^2$.  Consequently, the joint law and its position marginal
have total-variation error at most $\varepsilon$.

Choose $\tau=(Ld)^{-1/2}$.  The expected query count satisfies
\begin{equation}\label{eq:main-cold}
        Q_\nabla(\widehat T_\varepsilon^{\bps})
        \le 5\sqrt{Ld}\,\widehat T_\varepsilon^{\bps}
        =O\!\left(\kappa^{1/2}d\,
          \left(d\log\kappa+\log\frac1\varepsilon\right)\right).
\end{equation}
\end{theorem}

\begin{theorem}[Cold-start Zigzag complexity]\label{thm:zz-complexity}
Suppose that \eqref{eq:hessian-bounds} holds and initialize Zigzag according
to \eqref{eq:cold-start}.  Algorithm~\ref{alg:zz} is exact.  Set
$\gamma^{\zz}=\sqrt L$.  For $\varepsilon\in(0,1/2)$, define
\begin{equation}\label{eq:zz-horizon}
        \widehat T_\varepsilon^{\zz}
        :=\max\left\{
          \frac{K^{\zz}\sqrt L}{m}
          \log\!\left(
            1+\frac{K^{\zz}\chi^2(\rho_0\Vert\rho_\infty)}
                     {4\varepsilon^2}\right),
          L^{-1/2}\right\}.
\end{equation}
At time $\widehat T_\varepsilon^{\zz}$, the joint $\chi^2$ error is at most
$4\varepsilon^2$.  Consequently, the joint law and its position marginal
have total-variation error at most $\varepsilon$.

Choose  $\tau=L^{-1/2}d^{-1/4}$. The expected query counts satisfy
\begin{equation}\label{eq:zz-cold-complexity}
\begin{aligned}
        Q_\partial(\widehat T_\varepsilon^{\zz})
        &\le5\sqrt L\,d^{5/4}\widehat T_\varepsilon^{\zz}
        =O\!\left(\kappa d^{5/4}
             \left(d\log\kappa+\log\frac1\varepsilon\right)\right),\\
        Q_{\rm eq}(\widehat T_\varepsilon^{\zz})
        &\le5\sqrt L\,d^{1/4}\widehat T_\varepsilon^{\zz}
        =O\!\left(\kappa d^{1/4}
             \left(d\log\kappa+\log\frac1\varepsilon\right)\right).
\end{aligned}
\end{equation}
\end{theorem}

\begin{remark}[Comparison with Lu--Wang]\label{rem:zz-comparison}
Lu and Wang \cite[Theorem~1]{LuWangZigzag} analyze Zigzag query complexity
using the global envelope
\[
        \pos{v_i\partial_iU(x+tv)}
        \le L\abs{v_i}\bigl(\abs x+t\abs v\bigr),
\]
in the normalization $x_\star=0$.  Under their Assumption~2 and stated
dimension-and-accuracy regime, an $O(1)$-$\chi^2$ warm start yields the
high-probability bound
$O(\kappa^2d^{3/2}(\log\frac1\varepsilon)^{3/2})$ coordinate-partial
queries, or
$O(\kappa^2d^{1/2}(\log(1/\varepsilon))^{3/2})$ full-gradient equivalents.
Our locally anchored envelope, together with direct control of the expected
event and proposal counts, gives
$\widetilde O(\kappa d^{5/4})$ full-gradient equivalents from the explicit
Gaussian cold start \eqref{eq:cold-start}. We note that \cite{LuWangZigzag} provides a
high-probability warm-start guarantee under some additional assumptions, while our result only provides guarantee of expected complexity.
\end{remark}

\section{Proofs}\label{sec:proof}

The convergence estimates above determine how long the process must be run,
but not the cost of producing its trajectory.  For the windowed thinning
schemes, this cost comes from the gradient evaluations at the window anchors
and at the proposed event times.  The proposals themselves consist of
accepted events, whose frequency is governed by the true PDMP intensities,
and rejections, whose frequency is governed by the gap between those
intensities and their windowed envelopes.

The finite-horizon control needed for this analysis is common to both
samplers.  Bounces and coordinate flips preserve the speed, while
refreshments redraw it, so the refreshment history provides a pathwise
upper bound for the position and velocity.  Besides controlling the envelope
error, this bound supplies the integrability needed for the Dynkin identities
used to estimate the accepted-event counts and rules out explosion of the
proposal processes.  We then treat BPS and Zigzag separately, according to
the geometry of their jump mechanisms.  

\subsection{Common estimates}\label{ssec:common-estimates}

On $[0,T]$, for both processes, the speed can change only when the velocity is refreshed.  Let
$N_T^{\rm ref}$ denote the number of refreshments and $Z_j$ the Gaussian
velocity drawn at the $j$th refreshment.  Define
\[
        V_T^{\max}:=\max\left\{\abs{V_0},
                    \max_{1\le j\le N_T^{\rm ref}}\abs{Z_j}\right\},
\]
with the inner maximum taken to be zero when $N_T^{\rm ref}=0$.

\begin{lemma}[Trajectory and velocity-moment bounds]
\label{lem:trajectory}
Let $(X_t,V_t)_{t\ge0}$ be either BPS or Zigzag, and suppose that
$\abs{X_0}+\abs{V_0}<\infty$ almost surely.  Then, for every $T<\infty$,
\begin{equation}\label{eq:trajectory-bounds}
        \sup_{0\le t\le T}\abs{V_t}\le V_T^{\max},
        \qquad
        \sup_{0\le t\le T}\abs{X_t-x_\star}
        \le\abs{X_0-x_\star}+T V_T^{\max}.
\end{equation}
If $V_0\sim N(0,I_d)$, then, for every $T<\infty$ and $r>0$,
\begin{equation}\label{eq:vmax-moments}
        \E[(V_T^{\max})^r]<\infty,
\end{equation}
and, for every $t\ge0$,
\begin{equation}\label{eq:velocity-second-moment}
        \E\abs{V_t}^2=d.
\end{equation}
In the Zigzag case, one also has, for every $t\ge0$,
\begin{equation}\label{eq:zz-velocity-moments}
        \E\abs{V_t}_1^2
        =d+\frac{2}{\pi}d(d-1)\le d^2.
\end{equation}
If the process is initialized from the cold start \eqref{eq:cold-start}, then,
for every $t\ge0$,
\begin{equation}\label{eq:position-second-moment}
        \bigl(\E\abs{X_t-x_\star}^2\bigr)^{1/2}
        \le\sqrt{d/L}+\sqrt d\,t.
\end{equation}
\end{lemma}

\begin{proof}
Bounces and coordinate flips preserve the velocity norm, so only refreshments
can change it.  This proves the first bound in
\eqref{eq:trajectory-bounds}; integrating $\dot X_t=V_t$ proves the second.
For \eqref{eq:vmax-moments}, $N_T^{\rm ref}$ is Poisson and
$(V_T^{\max})^r\le\abs{V_0}^r+\sum_{j\le N_T^{\rm ref}}\abs{Z_j}^r$.
The refreshment draws are independent of the clock, so the expectation of the
sum is
$\E N_T^{\rm ref}\,\E\abs{Z_1}^r<\infty$, proving the claim.

For either process, every non-refreshment update preserves $\abs V$, whereas
each refreshment redraws $V\sim N(0,I_d)$.  Since $V_0\sim N(0,I_d)$,
$\abs{V_t}$ has the $\chi_d$ law at every deterministic time, which gives
\eqref{eq:velocity-second-moment}.

For Zigzag, coordinate flips leave
$(\abs{V_{1,t}},\ldots,\abs{V_{d,t}})$ unchanged.  At each refreshment this
vector is redrawn as the coordinatewise absolute value of an $N(0,I_d)$
vector.  Since $V_0\sim N(0,I_d)$, it has this distribution at every
deterministic time $t\ge0$.  Hence, with $Z\sim N(0,I_d)$,
\[
        \E\abs{V_t}_1^2
        =\sum_i\E Z_i^2+2\sum_{i<j}\E\abs{Z_i}\E\abs{Z_j}
        =d+\frac{2}{\pi}d(d-1),
\]
which proves \eqref{eq:zz-velocity-moments}.

Under the cold start,
\[
        \bigl(\E\abs{X_0-x_\star}^2\bigr)^{1/2}=\sqrt{d/L}.
\]
Moreover, since
$X_t-x_\star=X_0-x_\star+\int_0^tV_s\dd s$, we have
\[
\begin{aligned}
        \bigl(\E\abs{X_t-x_\star}^2\bigr)^{1/2}
        &\le
        \bigl(\E\abs{X_0-x_\star}^2\bigr)^{1/2}
        +\left(\E\left|\int_0^tV_s\dd s\right|^2\right)^{1/2}\\
        &\le \sqrt{d/L}
        +\int_0^t\bigl(\E\abs{V_s}^2\bigr)^{1/2}\dd s\\
        &=\sqrt{d/L}+\sqrt d\,t,
\end{aligned}
\]
where the last equality uses \eqref{eq:velocity-second-moment}.  This proves
\eqref{eq:position-second-moment} and completes the proof.
\end{proof}

Under the cold start, Lemma~\ref{lem:trajectory} bounds the position and
velocity on every finite horizon by a random variable with moments of all
orders.  This domination also justifies Dynkin’s formula for the unbounded test functions used in the event-count arguments.  We isolate the
required localization in the following technical lemma, which will be invoked
in both sampler-specific proofs below.

\begin{lemma}[Dynkin formula for observables of polynomial growth]
\label{lem:dynkin-polynomial}
Assume \eqref{eq:hessian-bounds} and the cold start
\eqref{eq:cold-start}, and consider either BPS or Zigzag, with generator
$\mathcal L$.  Let $f\in C^1(\mathbb R^{2d})$, and assume that both $f$ and $\mathcal Lf$ have polynomial growth; that is, for some $C<\infty$ and $p\ge0$,
\[
        \abs{f(x,v)}+\abs{\mathcal Lf(x,v)}
        \le C\bigl(1+\abs{x-x_\star}+\abs v\bigr)^p.
\]
Then, for every $T<\infty$,
\[
        \E f(X_T,V_T)-\E f(X_0,V_0)
        =\int_0^T\E[\mathcal Lf(X_t,V_t)]\dd t.
\]
\end{lemma}

\begin{proof}
Define stopping times
\[
        \sigma_n:=\inf\{t\ge0:
        \abs{X_t-x_\star}+\abs{V_t}\ge n\}.
\]
The jump rates are bounded on compact sets, and the Gaussian refreshment kernel preserves integrability of polynomial-growth observables, so the stopped process satisfies
Dynkin's formula:
\[
        \E f(X_{T\wedge\sigma_n},V_{T\wedge\sigma_n})
        -\E f(X_0,V_0)
        =\E\int_0^{T\wedge\sigma_n}
          \mathcal Lf(X_t,V_t)\dd t.
\]
For a fixed horizon $T$, Lemma~\ref{lem:trajectory} gives
\[
        \sup_{0\le t\le T}
        \bigl(\abs{X_t-x_\star}+\abs{V_t}\bigr)
        \le \abs{X_0-x_\star}+(T+1)V_T^{\max}.
\]
The right-hand side has moments of every order, and $\sigma_n>T$ eventually
almost surely.  The assumed polynomial growth bound therefore supplies a common
integrable majorant for both sides of the stopped identity.  Letting
$n\to\infty$ and applying dominated convergence proves the claim.
\end{proof}

\subsection{BPS}\label{ssec:proof-bps}

By Cauchy--Schwarz, controlling the expected number of bounces reduces to
controlling the time integral of the squared bounce rate.  To estimate this
integral using Dynkin's formula, we seek an observable whose generator
contains the squared bounce rate up to a spatial factor.  For an observable
$\psi$, the bounce contribution in Dynkin's formula is
\[
        \lambda^{\bps}(x,v)
        \bigl(\psi(x,R_xv)-\psi(x,v)\bigr),
\]
so the jump of $\psi$ should produce the factor
$\ip{v}{\nabla U(x)}$, since multiplication by $\lambda^{\bps}(x,v)$
then gives $\lambda^{\bps}(x,v)^2$.
For this purpose, we choose
\[
        \psi(x,v):=\ip{x-x_\star}{v}.
\]
When $\nabla U(x)\ne0$, the reflection formula gives
\begin{equation}\label{eq:bps-motivating-jump}
        \psi(x,R_xv)-\psi(x,v)
        =-2\ip{v}{\nabla U(x)}
          \frac{\ip{x-x_\star}{\nabla U(x)}}
               {\abs{\nabla U(x)}^2}.
\end{equation}
We center at $x_\star$ so that the smoothness and convexity assumptions bound
the spatial ratio on the right-hand side uniformly from below.  This
observation is the key input to the following estimate.

\begin{lemma}[Bound on the BPS bounce count from a cold start]
\label{lem:bps-count}
Assume \eqref{eq:hessian-bounds}.
For the cold start \eqref{eq:cold-start} and every $T\ge0$,
\begin{equation}\label{eq:bps-square-rate-bound}
        \int_0^T\E\bigl[\lambda^{\bps}(X_t,V_t)^2\bigr]\dd t
        \;\le\; LdT+\frac{\gamma^{\bps} d}{4}+\frac{\sqrt L\,d}{2}.
\end{equation}
Hence, whenever $T\ge\max\{L^{-1/2},\,\gamma^{\bps}/(4L)\}$, the expected
number of bounces in $[0,T]$ satisfies
\begin{equation}\label{eq:bps-count-clean}
        B^{\bps}(T)
        =\int_0^T\E\bigl[\lambda^{\bps}(X_t,V_t)\bigr]\dd t
        \le 2\sqrt{Ld}\,T .
\end{equation}
\end{lemma}

\begin{proof}
Let us introduce the shorthand
\[
        r(x):=
        \begin{cases}
        \displaystyle
        \frac{\ip{x-x_\star}{\nabla U(x)}}{\abs{\nabla U(x)}^2},
                &\nabla U(x)\ne0,\\[2mm]
        1/L,    &\nabla U(x)=0.
        \end{cases}
\]
With this notation, \eqref{eq:bps-motivating-jump} can be written as
\[
        \psi(x,R_xv)-\psi(x,v)
        =-2\ip{v}{\nabla U(x)}r(x).
\]
Strong convexity implies that $\nabla U(x)=0$ only at $x_\star$, where the
bounce rate vanishes; the assigned value at $x_\star$ is therefore not important.  

We claim that
\begin{equation}\label{eq:bps-psi-generator}
        \mathcal L^{\bps}\psi(x,v)
        =\abs v^2-2\lambda^{\bps}(x,v)^2r(x)
          -\gamma^{\bps}\psi(x,v).
\end{equation}
Indeed, since $\nabla_x\psi(x,v)=v$, the transport contribution in
\eqref{eq:bps-generator} is
\[
        v\cdot\nabla_x\psi(x,v)=\abs v^2.
\]
The bounce term follows from the calculation above. Finally, since $\varphi$ is centered,
\[
        \int_{\R^d}\psi(x,w)\varphi(\dd w)
        =\ip{x-x_\star}{\int_{\R^d}w\varphi(\dd w)}
        =0,
\]
so the refreshment contribution is
$-\gamma^{\bps}\psi(x,v)$.  Combining the three contributions gives
\eqref{eq:bps-psi-generator}.

The identity \eqref{eq:bps-psi-generator} involves
the squared bounce rate weighted by $r$.  To recover the unweighted rate, we
need a uniform lower bound on this factor.  We claim that
\[
        \frac1L\le r(x)\le\frac1m,
        \qquad x\in\R^d.
\]
The claim holds at $x_\star$ by the definition of $r$.  Suppose now that
$x\ne x_\star$.  Since $\nabla U(x_\star)=0$, the fundamental theorem of
calculus gives
\[
        \nabla U(x)=A_x(x-x_\star),
        \qquad
        A_x:=\int_0^1
        \nabla^2U\bigl(x_\star+s(x-x_\star)\bigr)\dd s.
\]
$A_x$ is symmetric and satisfies
$mI_d\preceq A_x\preceq L I_d$ by the bound on Hessian of $U$.  Its spectrum is therefore contained in
$[m,L]$, and hence $mA_x\preceq A_x^2\preceq LA_x$. Thus we arrived at the desired bound for $r$ combined with the observation that 
\[
\begin{aligned}
        r(x)
        =\frac{\ip{x-x_\star}{\nabla U(x)}}
                {\abs{\nabla U(x)}^2}
        =\frac{\ip{x-x_\star}{A_x(x-x_\star)}}
                {\ip{x-x_\star}{A_x^2(x-x_\star)}}.
\end{aligned}
\]

We may now apply Dynkin's formula.  The bounds
$\lambda^{\bps}(x,v)\le L\abs{x-x_\star}\abs v$ and
$r(x)\le1/m$ show from \eqref{eq:bps-psi-generator} that
$\abs\psi+\abs{\mathcal L^{\bps}\psi}$ has polynomial growth.  Therefore
Lemma~\ref{lem:dynkin-polynomial} applies and yields
\[
\begin{aligned}
        \E\psi(X_T,V_T)-\E\psi(X_0,V_0)
        & =\int_0^T
          \E\!\left[\mathcal L^{\bps}\psi(X_t,V_t)\right]\dd t \\
        &=\int_0^T\E\abs{V_t}^2\dd t
          -2\int_0^T\E\!\left[
            \lambda^{\bps}(X_t,V_t)^2r(X_t)
          \right]\dd t\\
        &\quad-\gamma^{\bps}
          \int_0^T\E\psi(X_t,V_t)\dd t,
\end{aligned}
\]
where the second equality uses  \eqref{eq:bps-psi-generator}.
By \eqref{eq:velocity-second-moment}, the first term on the right is $dT$.
Rearranging the preceding identity therefore gives
\begin{equation}\label{eq:dynkin}
\begin{aligned}
        2\int_0^T\E\!\left[
          \lambda^{\bps}(X_t,V_t)^2r(X_t)
        \right]\dd t
        &=dT-\gamma^{\bps}\int_0^T\E\psi(X_t,V_t)\dd t\\
        &\quad-\E\psi(X_T,V_T)+\E\psi(X_0,V_0).
\end{aligned}
\end{equation}
We next eliminate the time integral of $\psi$.  The position is continuous,
and $\dot X_t=V_t$ between events.  Therefore, for almost every $t$,
\[
        \frac{\dd}{\dd t}\abs{X_t-x_\star}^2
        =2\ip{X_t-x_\star}{V_t}
        =2\psi(X_t,V_t).
\]
Integration over $[0,T]$ and taking expectation therefore gives 
\[
        \int_0^T\E\psi(X_t,V_t)\dd t
        =\frac12\left(
          \E\abs{X_T-x_\star}^2-\E\abs{X_0-x_\star}^2
        \right).
\]
Consequently, the time-integral term in the rearranged Dynkin identity is
\[
        -\gamma^{\bps}\int_0^T\E\psi(X_t,V_t)\dd t
        =-\frac{\gamma^{\bps}}2
          \left(\E\abs{X_T-x_\star}^2-\E\abs{X_0-x_\star}^2\right).
\]
We now estimate this term together with the two endpoint terms on the right hand side of 
identity \eqref{eq:dynkin}. Under the cold start, $X_0-x_\star$ and $V_0$ are independent and centered, and hence
\[
        \E\psi(X_0,V_0)
        =\E\ip{X_0-x_\star}{V_0}=0,
        \qquad
        \E\abs{X_0-x_\star}^2=\frac dL.
\]
Since $\E\abs{X_T-x_\star}^2\ge0$, it follows that
\[
        -\gamma^{\bps}\int_0^T\E\psi(X_t,V_t)\dd t
        \le\frac{\gamma^{\bps}d}{2L}.
\]
For the remaining endpoint term, Cauchy--Schwarz,
\eqref{eq:velocity-second-moment}, and
\eqref{eq:position-second-moment} give
\[
\begin{aligned}
        -\E\ip{X_T-x_\star}{V_T}
        &\le
        \left|\E\ip{X_T-x_\star}{V_T}\right|\\
        &\le
        \sqrt{\E\abs{X_T-x_\star}^2\,\E\abs{V_T}^2}
        \le\left(\sqrt{\frac dL}+\sqrt d\,T\right)\sqrt d =\frac d{\sqrt L}+dT.
\end{aligned}
\]
Inserting these three bounds into \eqref{eq:dynkin} yields
\[
\begin{aligned}
        2\int_0^T\E\!\left[
          \lambda^{\bps}(X_t,V_t)^2r(X_t)
        \right]\dd t
        &\le 2dT+\frac{\gamma^{\bps}d}{2L}+\frac d{\sqrt L},
\end{aligned}
\]
and the lower bound $r\ge1/L$ then gives
\[
        \int_0^T
        \E\bigl[\lambda^{\bps}(X_t,V_t)^2\bigr]\dd t
        \le LdT+\frac{\gamma^{\bps}d}{4}+\frac{\sqrt L\,d}{2},
\]
which proves \eqref{eq:bps-square-rate-bound}.

It remains to deduce the bounce-count estimate. Cauchy--Schwarz gives
\[
        B^{\bps}(T)
        =\int_0^T\E\lambda^{\bps}(X_t,V_t)\dd t
        \le\left(
          T\int_0^T\E[\lambda^{\bps}(X_t,V_t)^2]\dd t
        \right)^{1/2}.
\]
If $T\ge L^{-1/2}$ and $T\ge\gamma^{\bps}/(4L)$, then
\[
        \frac{\sqrt L\,d}{2}\le\frac12LdT,
        \qquad
        \frac{\gamma^{\bps}d}{4}\le LdT.
\]
Thus \eqref{eq:bps-square-rate-bound} implies
\[
        \int_0^T\E[\lambda^{\bps}(X_t,V_t)^2]\dd t
        \le\frac52LdT.
\]
Substitution into the Cauchy--Schwarz bound gives
\[
        B^{\bps}(T)
        \le\sqrt{\frac52}\,\sqrt{Ld}\,T
        \le2\sqrt{Ld}\,T,
\]
which proves \eqref{eq:bps-count-clean}.
\end{proof}

The bounce count controls the accepted proposals.  It remains to count the
window anchors and the rejected thinning proposals.

\begin{prop}[BPS windowed thinning: exactness and query count]
\label{prop:bps-thinning}
Assume \eqref{eq:hessian-bounds} and let $T,\tau>0$.
Run Algorithm~\ref{alg:bps} with refreshment rate $\gamma^{\bps}>0$ and a
possibly random initial state $(X_0,V_0)$ satisfying
$\abs{X_0}+\abs{V_0}<\infty$ almost surely.

Then Algorithm~\ref{alg:bps} is well defined and exactly simulates the BPS
process initialized at $(X_0,V_0)$ with refreshment rate $\gamma^{\bps}$.
Denote the process generated by the algorithm by $(X_t,V_t)$. It satisfies
\begin{equation}\label{eq:bps-thinning-general}
Q_\nabla(T)
\le \frac{T}{\tau}+1+B^{\bps}(T)
+L\tau T\sup_{0\le t\le T}\E\abs{V_t}^2.
\end{equation}
For the cold start \eqref{eq:cold-start}, choosing
$\tau=(Ld)^{-1/2}$ gives
\begin{equation}\label{eq:bps-thinning-cold}
Q_\nabla(T)\le 5\sqrt{Ld}\,T
\end{equation}
whenever
$T\ge\max\{L^{-1/2},\gamma^{\bps}/(4L)\}$.
\end{prop}

\begin{proof}
The calculation preceding \eqref{eq:bps-envelope} shows that the envelope
dominates the BPS rate throughout each window.
Let us show that the proposal process is nonexplosive.  Since bounces
preserve the speed and only refreshments can change it, for $t\le T$,
\[
        \abs{V_{t-}}\le V_T^{\max},
        \qquad
        D_t\le T V_T^{\max},
        \qquad
        \abs{X_{t_k}-x_\star}
        \le\abs{X_0-x_\star}+T V_T^{\max}.
\]
Using $\abs{\nabla U(x)}\le L\abs{x-x_\star}$, we therefore obtain
\[
        \abs{G_k}
        =\abs{\nabla U(X_{t_k})}
        \le L\bigl(\abs{X_0-x_\star}+T V_T^{\max}\bigr).
\]
Substituting these bounds into \eqref{eq:bps-envelope} yields
\[
\begin{aligned}
        \bar\lambda_t
        &\le \abs{V_{t-}}\abs{G_k}
             +L\abs{V_{t-}}D_t\\
        &\le LV_T^{\max}
             \bigl(\abs{X_0-x_\star}+2TV_T^{\max}\bigr)<\infty.
\end{aligned}
\]
Thus the proposal intensity is bounded on $[0,T]$, which proves
nonexplosion, and thus the simulation is well-defined and exact.

We next count its gradient queries.  There are at most $T/\tau+1$ nonempty
windows and hence at most that many anchor queries.  Since every proposal
requires one further gradient query, we have
\[
\begin{aligned}
        Q_\nabla(T)
        &\le\frac{T}{\tau}+1+\int_0^T\E\bar\lambda_t\dd t\\
        &=\frac{T}{\tau}+1+B^{\bps}(T)
          +\int_0^T\E\!\left[
            \bar\lambda_t-\lambda^{\bps}(X_{t-},V_{t-})
          \right]\dd t.
\end{aligned}
\]
The last integral is the expected number of rejected proposals.
To control this term, we compare the anchored part of the envelope with the
true bounce rate.  Since $\pos{\cdot}$ is $1$-Lipschitz,
\[
\begin{aligned}
        0\le{}
        \bar\lambda_t-\lambda^{\bps}(X_{t-},V_{t-})
        &=\pos{\ip{V_{t-}}{G_k}}
          -\pos{\ip{V_{t-}}{\nabla U(X_{t-})}}
          +L\abs{V_{t-}}D_t\\
        &\le
          \abs{\ip{V_{t-}}{G_k-\nabla U(X_{t-})}}
          +L\abs{V_{t-}}D_t\\
        &\le2L\abs{V_{t-}}D_t.
\end{aligned}
\]
The last inequality uses
$\abs{G_k-\nabla U(X_{t-})}\le L D_t$ following the smoothness of $U$.
For $t\in[t_k,t_{k+1})$, write $s=t-t_k$.  Since
$D_t=\int_{t_k}^t\abs{V_u}\dd u$,
\[
\begin{aligned}
        \E[\abs{V_{t-}}D_t]
        &=\int_{t_k}^t\E[\abs{V_{t-}}\abs{V_u}]\dd u\\
        &\le\int_{t_k}^t
          \sqrt{\E\abs{V_{t-}}^2\,\E\abs{V_u}^2}\dd u\\
        &\le s\sup_{0\le u\le T}\E\abs{V_u}^2.
\end{aligned}
\]
If the $k$th nonempty window has length $\ell_k\le\tau$, its contribution to
the rejected-proposal count is therefore at most
\[
\begin{aligned}
        \int_{t_k}^{t_k+\ell_k}
        \E\!\left[
          \bar\lambda_t-\lambda^{\bps}(X_{t-},V_{t-})
        \right]\dd t
        &\le2L\left(\sup_{0\le u\le T}\E\abs{V_u}^2\right)
          \int_0^{\ell_k}s\dd s\\
        &=L\ell_k^2
          \sup_{0\le u\le T}\E\abs{V_u}^2.
\end{aligned}
\]
Summing over the windows and using
$\sum_k\ell_k^2\le\tau\sum_k\ell_k=\tau T$ gives
\[
        \int_0^T\E\!\left[
          \bar\lambda_t-\lambda^{\bps}(X_{t-},V_{t-})
        \right]\dd t
        \le L\tau T\sup_{0\le u\le T}\E\abs{V_u}^2.
\]
Substitution into the query decomposition proves
\eqref{eq:bps-thinning-general}.

For the cold start, \eqref{eq:velocity-second-moment} gives
$\sup_{0\le t\le T}\E\abs{V_t}^2=d$, while
Lemma~\ref{lem:bps-count} gives
$B^{\bps}(T)\le2\sqrt{Ld}\,T$ on the stated horizon.  With
$\tau=(Ld)^{-1/2}$, \eqref{eq:bps-thinning-general} becomes
\[
\begin{aligned}
        Q_\nabla(T)
        &\le\sqrt{Ld}\,T+1
          +2\sqrt{Ld}\,T
          +L(Ld)^{-1/2}Td\\
        &=4\sqrt{Ld}\,T+1
        \le5\sqrt{Ld}\,T.
\end{aligned}
\]
This proves \eqref{eq:bps-thinning-cold}.
\end{proof}

\begin{proof}[Proof of Theorem~\ref{thm:main}]
The cold start satisfies the initial-state assumption in
Proposition~\ref{prop:bps-thinning}, and \eqref{eq:bps-horizon} includes the
lower bounds required for \eqref{eq:bps-thinning-cold}.  The proposition
therefore shows that
Algorithm~\ref{alg:bps} is well defined on
$[0,\widehat T_\varepsilon^{\bps}]$, simulates BPS exactly, and satisfies
\eqref{eq:bps-thinning-cold}.

The choice of $\widehat T_\varepsilon^{\bps}$  and Theorem~\ref{thm:LW} give
\[
\begin{aligned}
&\chi^2\!\left(
  \Law(X_{\widehat T_\varepsilon^{\bps}},
       V_{\widehat T_\varepsilon^{\bps}})
  \,\middle\Vert\,\rho_\infty\right)\\
&\quad\le
K^{\bps}
\exp\!\left(
  -\frac{1}{K^{\bps}}\sqrt{\frac md}\,
   \widehat T_\varepsilon^{\bps}
\right)
\chi^2(\rho_0\Vert\rho_\infty) \le4\varepsilon^2,
\end{aligned}
\]
which gives the total-variation error $\varepsilon$ by Cauchy–Schwarz inequality.

By \eqref{eq:cold-start-chi2},
\[
\begin{aligned}
\log\!\left(
  1+
  \frac{\sqrt{K^{\bps}\chi^2(\rho_0\Vert\rho_\infty)}}
       {2\varepsilon}
\right)
&\le
\log\!\left(
  1+\frac{\sqrt{K^{\bps}}\,\kappa^{d/4}}{2\varepsilon}
\right)\\
&\le C\left(d\log\kappa+\log\frac1\varepsilon\right).
\end{aligned}
\]
The other two entries in \eqref{eq:bps-horizon} are also bounded by
$\sqrt{d/m}$, since $L^{-1/2}\le m^{-1/2}$ and
$\gamma^{\bps}/(4L)\le\kappa^{-1}\sqrt{d/m}$.
Therefore, absorbing some constants,
\[
        \widehat T_\varepsilon^{\bps}
        =O\!\left(\sqrt{\frac dm}\,
          \left(d\log\kappa+\log\frac1\varepsilon\right)\right).
\]
Since
$\sqrt{Ld}\sqrt{d/m}=d\sqrt\kappa$, the query estimate becomes
\[
        Q_\nabla(\widehat T_\varepsilon^{\bps})
        \le5\sqrt{Ld}\,\widehat T_\varepsilon^{\bps}
        =O\!\left(
          \kappa^{1/2}d
          \left(d\log\kappa+\log\frac1\varepsilon\right)\right),
\]
which proves \eqref{eq:main-cold}.
\end{proof}

\subsection{Zigzag}\label{ssec:proof-zz}
We follow a similar strategy as the analysis of BPS, while in the case of Zigzag, we would need to control the quadratic aggregated flip rate:
\[
        S(x,v):=\sum_{i=1}^d\lambda_i^{\zz}(x,v)^2.
\]
To estimate its time integral using Dynkin's formula, we choose the observable
\[
        q(x,v):=v\cdot\nabla U(x).
\]
A coordinate flip changes the corresponding summand of $q$ so that the jump
part of its generator produces $S$.  The remaining terms are controlled by
the Hessian bound and the moment estimates from
Lemma~\ref{lem:trajectory}.

\begin{lemma}[Bound on the Zigzag flip count from a cold start]
\label{lem:zz-count}
Assume \eqref{eq:hessian-bounds}.
Consider Zigzag initialized from the cold start \eqref{eq:cold-start}, with
arbitrary refreshment rate $\gamma^{\zz}>0$.  Then, for every $T\ge0$,
\begin{equation}\label{eq:zz-square-rate-bound}
        \int_0^T\E S(X_t,V_t)\dd t
        \le LdT+\frac{\gamma^{\zz}d}{4}+\frac{\sqrt L\,d}{2},
\end{equation}
and hence, whenever
$T\ge\max\{L^{-1/2},\gamma^{\zz}/(4L)\}$, the expected total number of flips
satisfies
\begin{equation}\label{eq:zz-count-clean}
        B^{\zz}(T)
        =\int_0^T\E\Lambda^{\zz}(X_t,V_t)\dd t
        \le2d\sqrt L\,T .
\end{equation}
\end{lemma}

\begin{proof}
We first compute the generator of $q$:
\begin{equation}\label{eq:zz-q-generator}
        \mathcal L^{\zz}q(x,v)
        =v^\top\nabla^2U(x)v-2S(x,v)-\gamma^{\zz}q(x,v).
\end{equation}
Indeed, the transport contribution is $v^\top\nabla^2U(x)v$.  For a flip of
coordinate $i$,
\[
        q(x,F_iv)-q(x,v)=-2v_i\partial_iU(x),
        \qquad
        \lambda_i^{\zz}(x,v)v_i\partial_iU(x)
        =\lambda_i^{\zz}(x,v)^2,
\]
so summing the flip terms gives $-2S(x,v)$.  The Gaussian velocity law is
centered, and hence the refreshment term is $-\gamma^{\zz}q(x,v)$.  This
proves \eqref{eq:zz-q-generator}.

Since $\nabla U(x_\star)=0$, the $L$-smoothness and mean-value theorem give
$\abs{\nabla U(x)}\le L\abs{x-x_\star}$.  Consequently,
\[
\begin{aligned}
        \abs{q(x,v)}
        &\le\abs v\,\abs{\nabla U(x)}
        \le L\abs{x-x_\star}\abs v,\\
        S(x,v)
        &\le\sum_{i=1}^d v_i^2\bigl(\partial_iU(x)\bigr)^2
        \le\abs v^2\abs{\nabla U(x)}^2
        \le L^2\abs{x-x_\star}^2\abs v^2,\\
        \abs{v^\top\nabla^2U(x)v}
        &\le\|\nabla^2U(x)\|_{\mathrm{op}}\abs v^2
        \le L\abs v^2.
\end{aligned}
\]
These estimates show that $q$ and $\mathcal L^{\zz}q$ have polynomial growth.
We may thus
apply Lemma~\ref{lem:dynkin-polynomial}.  Using
\eqref{eq:zz-q-generator}, Dynkin's formula gives
\[
\begin{aligned}
        \E q(X_T,V_T)-\E q(X_0,V_0)
        &=\int_0^T
          \E\bigl[V_t^\top\nabla^2U(X_t)V_t\bigr]\dd t\\
        &\quad-2\int_0^T\E S(X_t,V_t)\dd t
          -\gamma^{\zz}\int_0^T\E q(X_t,V_t)\dd t.
\end{aligned}
\]

To express the last integral in terms of the potential, note that $X_t$ is
continuous and $\dot X_t=V_t$ for almost every $t$.  The chain rule along
each trajectory gives
\[
        U(X_T)-U(X_0)
        =\int_0^T\nabla U(X_t)\cdot V_t\dd t
        =\int_0^Tq(X_t,V_t)\dd t.
\]
Substituting this identity into Dynkin's
formula and rearranging yields
\[
\begin{aligned}
        2\int_0^T\E S(X_t,V_t)\dd t
        &=\int_0^T\E\bigl[V_t^\top\nabla^2U(X_t)V_t\bigr]\dd t \\
        &\quad-\gamma^{\zz}\E\bigl[U(X_T)-U(X_0)\bigr]
              -\E q(X_T,V_T)+\E q(X_0,V_0).
\end{aligned}
\]

Since $V_0$ is centered and independent of $X_0$,
$\E q(X_0,V_0)=0$.  The Hessian bound and
\eqref{eq:velocity-second-moment} imply
\[
        \int_0^T\E\bigl[V_t^\top\nabla^2U(X_t)V_t\bigr]\dd t
        \le L\int_0^T\E\abs{V_t}^2\dd t
        =LdT.
\]
Since $U(X_T)\ge U(x_\star)$, the upper Hessian bound and the cold start give
\[
\begin{aligned}
        -\E[U(X_T)-U(X_0)]
        &\le\E[U(X_0)-U(x_\star)]
        \le\frac{L}{2}
             \E\abs{X_0-x_\star}^2
        =\frac{d}{2}.
\end{aligned}
\]
For the terminal term, Lemma~\ref{lem:trajectory} yields
\[
\begin{aligned}
        -\E q(X_T,V_T)
        &\le\left|\E q(X_T,V_T)\right|\\
        &\le L
          \sqrt{\E\abs{X_T-x_\star}^2\,\E\abs{V_T}^2}\\
        &\le L\left(\sqrt{\frac dL}+\sqrt d\,T\right)\sqrt d=\sqrt L\,d+LdT.
\end{aligned}
\]
Inserting these bounds into the Dynkin identity gives
\[
        2\int_0^T\E S(X_t,V_t)\dd t
        \le 2LdT+\frac{\gamma^{\zz}d}{2}+\sqrt L\,d,
\]
which proves \eqref{eq:zz-square-rate-bound}.

Finally, $(\Lambda^{\zz})^2\le dS$, so Cauchy--Schwarz on
$[0,T]\times\Omega$ gives
\[
\begin{aligned}
        B^{\zz}(T)
        &\le\left(
          T\int_0^T\E[\Lambda^{\zz}(X_t,V_t)^2]\dd t
        \right)^{1/2}\\
        &\le\left(
          dT\int_0^T\E S(X_t,V_t)\dd t
        \right)^{1/2}.
\end{aligned}
\]
Under the stated lower bound on $T$, the last two terms in
\eqref{eq:zz-square-rate-bound} are at most $LdT$ and $LdT/2$,
respectively.  Hence
\[
        \int_0^T\E S(X_t,V_t)\dd t\le\frac52LdT,
        \qquad
        B^{\zz}(T)
        \le\sqrt{\frac52}\,d\sqrt L\,T
        \le2d\sqrt L\,T,
\]
which proves \eqref{eq:zz-count-clean}.
\end{proof}

The flip count controls the accepted proposals.  It remains to count the
window anchors and the rejected coordinate proposals.

\begin{prop}[Zigzag windowed thinning: exactness and query count]
\label{prop:zz-thinning}
Assume \eqref{eq:hessian-bounds} and let $T,\tau>0$.
Run Algorithm~\ref{alg:zz} with refreshment rate $\gamma^{\zz}>0$ and a
possibly random initial state $(X_0,V_0)$ satisfying  $\abs{X_0}+\abs{V_0}<\infty$ almost surely.

Then Algorithm~\ref{alg:zz} is well defined and exactly simulates the Zigzag
process initialized at $(X_0,V_0)$ with refreshment rate $\gamma^{\zz}$.
Denote the process generated by the algorithm by $(X_t,V_t)$. It satisfies
\begin{equation}\label{eq:zz-thinning-general}
Q_\partial(T)
\le d\left(\frac{T}{\tau}+1\right)
+B^{\zz}(T)
+L\tau T
\left(\sup_{0\le t\le T}\E\abs{V_t}_1^2\right)^{1/2}
\left(\sup_{0\le t\le T}\E\abs{V_t}^2\right)^{1/2}.
\end{equation}
If $V_0\sim N(0,I_d)$, then
\begin{equation}\label{eq:zz-thinning-equivalent}
Q_{\rm eq}(T)
\le\frac{T}{\tau}+1+\frac{B^{\zz}(T)}{d}
+L\sqrt d\,\tau T.
\end{equation}

For the cold start \eqref{eq:cold-start}, setting
$\gamma^{\zz}=\sqrt L$ and choosing
$\tau=L^{-1/2}d^{-1/4}$ gives
\begin{equation}\label{eq:zz-thinning-cold}
Q_{\rm eq}(T)\le5\sqrt L\,d^{1/4}T
\end{equation}
whenever $T\ge L^{-1/2}$.
\end{prop}

\begin{proof}
The calculation preceding \eqref{eq:zz-envelope} shows that every coordinate
envelope dominates its flip rate throughout each window.
Let us show that the proposal process is nonexplosive.  Since coordinate flips
preserve the speed and only refreshments can change it, for $t\le T$,
\[
        \abs{V_{t-}}\le V_T^{\max},
        \qquad
        D_t\le T V_T^{\max},
        \qquad
        \abs{X_{t_k}-x_\star}
        \le\abs{X_0-x_\star}+T V_T^{\max}.
\]
Moreover, $\abs{V_{t-}}_1\le\sqrt d\,V_T^{\max}$.  Using
$\abs{\nabla U(x)}\le L\abs{x-x_\star}$, we therefore obtain
\[
        \abs{G_k}
        =\abs{\nabla U(X_{t_k})}
        \le L\bigl(\abs{X_0-x_\star}+T V_T^{\max}\bigr).
\]
Substituting these bounds into \eqref{eq:zz-envelope} yields
\[
\begin{aligned}
        \bar\Lambda_t
        &\le\abs{V_{t-}}\abs{G_k}
             +L\abs{V_{t-}}_1D_t\\
        &\le L V_T^{\max}
             \bigl(\abs{X_0-x_\star}+T V_T^{\max}\bigr)
             +\sqrt d\,LT(V_T^{\max})^2<\infty.
\end{aligned}
\]
Thus the total proposal intensity is bounded on $[0,T]$, which proves
nonexplosion, and thus the simulation is well-defined and exact.

We next count its coordinate-partial queries.  There are at most
$T/\tau+1$ nonempty windows and hence at most
$d(T/\tau+1)$ coordinate-partial anchor queries.  Since every proposal
requires one further coordinate-partial query, we have
\[
\begin{aligned}
        Q_\partial(T)
        &\le d\left(\frac{T}{\tau}+1\right)
             +\int_0^T\E\bar\Lambda_t\dd t\\
        &=d\left(\frac{T}{\tau}+1\right)+B^{\zz}(T)
          +\int_0^T\E\!\left[
            \bar\Lambda_t-\Lambda^{\zz}(X_{t-},V_{t-})
          \right]\dd t.
\end{aligned}
\]
The last integral is the expected number of rejected coordinate proposals.
To control this term, we compare the anchored part of each coordinate envelope
with its true flip rate.  Since $\pos{\cdot}$ is $1$-Lipschitz,
\[
\begin{aligned}
        0\le{}
        \bar\Lambda_t-\Lambda^{\zz}(X_{t-},V_{t-})
        &=\sum_{i=1}^d\left(
          \pos{V_{i,t-}G_{k,i}}
          -\pos{V_{i,t-}\partial_iU(X_{t-})}
          +L\abs{V_{i,t-}}D_t
        \right)\\
        &\le\sum_{i=1}^d\abs{V_{i,t-}}\,
          \abs{G_{k,i}-\partial_iU(X_{t-})}
          +L\abs{V_{t-}}_1D_t\\
        &\le2L\abs{V_{t-}}_1D_t.
\end{aligned}
\]
The last inequality uses
$\abs{G_{k,i}-\partial_iU(X_{t-})}
\le\abs{G_k-\nabla U(X_{t-})}\le L D_t$.
For $t\in[t_k,t_{k+1})$, write $s=t-t_k$.  Since
$D_t=\int_{t_k}^t\abs{V_u}\dd u$,
\[
\begin{aligned}
        \E[\abs{V_{t-}}_1D_t]
        &=\int_{t_k}^t
          \E[\abs{V_{t-}}_1\abs{V_u}]\dd u\\
        &\le\int_{t_k}^t
          \sqrt{\E\abs{V_{t-}}_1^2\,\E\abs{V_u}^2}\dd u\\
        &\le s
          \left(\sup_{0\le u\le T}\E\abs{V_u}_1^2\right)^{1/2}
          \left(\sup_{0\le u\le T}\E\abs{V_u}^2\right)^{1/2}.
\end{aligned}
\]
If the $k$th nonempty window has length $\ell_k\le\tau$, its contribution to
the rejected-proposal count is therefore at most
\[
\begin{aligned}
        \int_{t_k}^{t_k+\ell_k}
        \E\!\left[
          \bar\Lambda_t-\Lambda^{\zz}(X_{t-},V_{t-})
        \right]\dd t
        &\le2L
          \left(\sup_{0\le u\le T}\E\abs{V_u}_1^2\right)^{1/2}
          \left(\sup_{0\le u\le T}\E\abs{V_u}^2\right)^{1/2}
          \int_0^{\ell_k}s\dd s\\
        &=L\ell_k^2
          \left(\sup_{0\le u\le T}\E\abs{V_u}_1^2\right)^{1/2}
          \left(\sup_{0\le u\le T}\E\abs{V_u}^2\right)^{1/2}.
\end{aligned}
\]
Summing over the windows and using
$\sum_k\ell_k^2\le\tau\sum_k\ell_k=\tau T$ bounds the rejection term by
\[
        L\tau T
        \left(\sup_{0\le u\le T}\E\abs{V_u}_1^2\right)^{1/2}
        \left(\sup_{0\le u\le T}\E\abs{V_u}^2\right)^{1/2}.
\]
Substitution into the query decomposition proves
\eqref{eq:zz-thinning-general}.

If $V_0\sim N(0,I_d)$, then
\eqref{eq:velocity-second-moment} and \eqref{eq:zz-velocity-moments} give
\[
\left(\sup_{0\le t\le T}\E\abs{V_t}_1^2\right)^{1/2}
\left(\sup_{0\le t\le T}\E\abs{V_t}^2\right)^{1/2}
\le d\sqrt d.
\]
Combined with \eqref{eq:zz-thinning-general} now gives
\eqref{eq:zz-thinning-equivalent} as $Q_{\rm eq}(T) = \frac{1}{d} Q_\partial(T)$.

For the cold start, $\gamma^{\zz}=\sqrt L$ and $T\ge L^{-1/2}$ imply
$T\ge\gamma^{\zz}/(4L)$.  Lemma~\ref{lem:zz-count} therefore gives
$B^{\zz}(T)/d\le2\sqrt L\,T$.  With
$\tau=L^{-1/2}d^{-1/4}$, \eqref{eq:zz-thinning-equivalent} becomes
\[
\begin{aligned}
        Q_{\rm eq}(T)
        &\le\sqrt L\,d^{1/4}T+1
             +2\sqrt L\,T
             +\sqrt L\,d^{1/4}T\\
        &\le4\sqrt L\,d^{1/4}T+1
        \le5\sqrt L\,d^{1/4}T.
\end{aligned}
\]
The last inequality uses
$\sqrt L\,d^{1/4}T\ge1$.
This proves 
\eqref{eq:zz-thinning-cold}.
\end{proof}

\begin{proof}[Proof of Theorem~\ref{thm:zz-complexity}]
The cold start satisfies the initial-state assumption in
Proposition~\ref{prop:zz-thinning}, and \eqref{eq:zz-horizon} includes the
lower bound required for \eqref{eq:zz-thinning-cold}.  The proposition
therefore shows that Algorithm~\ref{alg:zz} is well defined on
$[0,\widehat T_\varepsilon^{\zz}]$, simulates Zigzag exactly, and satisfies
\eqref{eq:zz-thinning-cold}.

The choice of $\widehat T_\varepsilon^{\zz}$ and
Theorem~\ref{thm:LW-ZZ} yield
\[
\begin{aligned}
&\chi^2\!\left(
  \Law(X_{\widehat T_\varepsilon^{\zz}},
       V_{\widehat T_\varepsilon^{\zz}})
  \,\middle\Vert\,\rho_\infty\right)\\
&\quad\le
K^{\zz}
\exp\!\left(
  -\frac{m}{K^{\zz}\sqrt L}\,
   \widehat T_\varepsilon^{\zz}
\right)
\chi^2(\rho_0\Vert\rho_\infty)
\le4\varepsilon^2.
\end{aligned}
\]
This gives the total-variation error $\varepsilon$ by Cauchy–Schwarz inequality.

By \eqref{eq:cold-start-chi2},
\[
\begin{aligned}
\log\!\left(
  1+\frac{K^{\zz}\chi^2(\rho_0\Vert\rho_\infty)}
           {4\varepsilon^2}
\right)
&\le
\log\!\left(
  1+\frac{K^{\zz}\kappa^{d/2}}{4\varepsilon^2}
\right)\\
&\le C\left(d\log\kappa+\log\frac1\varepsilon\right).
\end{aligned}
\]
The other entry in \eqref{eq:zz-horizon} is also bounded by
$\sqrt L/m$, since $L^{-1/2}\le\sqrt L/m$.  Therefore, absorbing some
constants,
\[
        \widehat T_\varepsilon^{\zz}
        =O\!\left(\frac{\sqrt L}{m}
          \left(d\log\kappa+\log\frac1\varepsilon\right)\right).
\]
Since $\sqrt L(\sqrt L/m)=\kappa$, the query estimates become
\[
\begin{aligned}
        Q_\partial(\widehat T_\varepsilon^{\zz})
        &\le5\sqrt L\,d^{5/4}\widehat T_\varepsilon^{\zz}
        =O\!\left(
          \kappa d^{5/4}
          \left(d\log\kappa+\log\frac1\varepsilon\right)\right),\\
        Q_{\rm eq}(\widehat T_\varepsilon^{\zz})
        &\le5\sqrt L\,d^{1/4}\widehat T_\varepsilon^{\zz}
        =O\!\left(
          \kappa d^{1/4}
          \left(d\log\kappa+\log\frac1\varepsilon\right)\right).
\end{aligned}
\]
This proves \eqref{eq:zz-cold-complexity}.
\end{proof}

\bibliographystyle{alpha}
\bibliography{references}

\end{document}